\documentclass[a4paper,reqno]{amsart}  

\usepackage[T1]{fontenc}         
\usepackage[applemac]{inputenc} 
\usepackage{amsthm}
\usepackage{amsmath}
\usepackage{amsfonts}
\usepackage{amssymb}
\usepackage[arrow, matrix, curve]{xy}
\usepackage{picinpar, graphicx}
\usepackage{geometry}
\usepackage{epic}
\usepackage{color}

\newtheorem{Thm}{Theorem}

\newtheorem{Lem}{Lemma}
\newtheorem{Prop}{Proposition}
\newtheorem{Cor}{Corollary}

\begin{document}

\renewcommand{\thefootnote}{}
\footnotetext{Research partially supported by NSERC Grant RGPIN 105490-2011 and by the Minist\'erio de Ci\^encia e Tecnologia, Brazil, CNPq Proc. No. 303774/2009-6.}

\title[Local Isometric Immersions of $\eta$ p.s.s. and $k$-th Order Evolution Equations]{Local Isometric Immersions of Pseudo-spherical Surfaces and $k$-th Order Evolution Equations}
\author{Nabil Kahouadji, Niky Kamran  and Keti Tenenblat}
\date{}

\begin{abstract}
We consider the class of evolution equations that  describe pseudo-spherical surfaces of the form $u_t=F(u,\partial u/\partial x,..., \partial^k u/\partial x^k)$, $k\geq 2$  classified by Chern-Tenenblat. 
This class of equations is characterized by the property that to each 
solution of a differential equation within this class, there  corresponds 
a 2-dimensional Riemannian metric of curvature -1.\linebreak  We investigate the 
following problem: given such a metric, is there a local isometric immersion   
in $\mathbb{R}^3$ such that  the coefficients of the second fundamental form of the surface depend on a jet of finite order of $u$?  
By extending our previous result for second order evolution equation to 
$k$-th order equations, we prove that there is only one type of equations that admit such an isometric immersion. We prove that the coefficients of the second fundamental forms of the local isometric immersion determined by the solutions $u$ are universal, i.e., they are independent of $u$. Moreover, we show that there exists a foliation of the domain of the parameters of the surface by straight lines with the property that the mean curvature of the surface is constant along the images of these straight lines under the isometric immersion.   

\end{abstract}

\maketitle

\section{Introduction}
This paper is the third in a series~\cite{KKT1,KKT2} in which we consider the special properties of the local isometric immersions into three-dimensional Euclidean space ${\mathbb{E}}^{3}$ of the metrics of constant negative Gaussian curvature $K=-1$ associated to the solutions $u$ of evolution equations
\begin{equation}\label{evol}
\frac{\partial u}{\partial t}=F(u,\frac{\partial u}{\partial x},\ldots,\frac{\partial^{k} u}{\partial x^{k}}),
\end{equation}
describing pseudo-spherical surfaces. For reasons that will be explained below, our main interest lies in determining evolution equations (\ref{evol}) for which the components of the second fundamental form of the local isometric immersion depend on $u$ and \emph{finitely many of its derivatives only}, in other words on a jet of finite order of $u$. 

Recall following Chern and Tenenblat~\cite{ChernTenenblat} that a partial differential equation 
\begin{equation}\label{pde}
\Delta\left(t,x,u,\frac{\partial u}{\partial x},\frac{\partial u}{\partial t},\ldots,\frac{\partial^{k}u}{\partial t^{l}\partial x^{k-l}}\right)=0,
\end{equation}
belongs to the class of differential equations describing pseudo-spherical surfaces if there exist $1$-forms 
\begin{equation}\label{forms}
\omega^{i}=f_{i1}dx+f_{i2}dt,\quad 1\leq i \leq3,
\end{equation}
where the coefficients $f_{ij},\,1\leq i \leq 3,\,1\leq j\leq 2,$ are smooth functions of $t,x,u$ and finitely many derivatives of $u$ with respect to $t$ and $x$, such that the structure equations 
\begin{equation}\label{struct}
d\omega^{1}=\omega^{3} \wedge\omega^{2},\quad d\omega^{2}=\omega^{1} \wedge\omega^{3},\quad d\omega^{3}=\omega^{1} \wedge\omega^{2}\neq 0,
\end{equation}
for a metric of constant Gaussian curvature $K=-1$  hold if, and only if, $u$ is a solution of (\ref{pde}). In this case every smooth solution $u: U\subset \mathbb{R}^{2} \to \mathbb{R} $ of an equation (\ref{pde}) describing pseudo-spherical surfaces defines on $U$ a Riemannian metric 
\begin{equation}\label{metric}
ds^{2}=(\omega^{1})^{2}+(\omega^{2})^{2},
\end{equation}
of constant Gaussian curvature $K=-1$, with $\omega^{3}$ being the Levi-Civita connection $1$-form of the metric (\ref{metric}).

An important motivation for the question investigated in this paper comes from the special properties of the sine-Gordon equation
\begin{equation}\label{sG}
u_{tx}=\sin u,
\end{equation}
whose well-known integrability properties can be completely accounted for through the perspective of the general theory developed by Chern and Tenenblat. First, it is straightforward to check that the $1$-forms 
\begin{equation}\label{sGcof1}
\omega_1 = \frac{1}{\eta}\sin u \, dt, \quad \omega_2 = \eta\, dx+\frac{1}{\eta}\cos u \,dt,\quad \omega_3 = u_{x}\,dx.
\end{equation}
satisfy the structure equations (\ref{struct}) whenever $u$ is a solution of the sine-Gordon equation (\ref{sG}). The non-zero real parameter $\eta$ appearing in (\ref{sGcof1}) is directly related to the existence of a one-parameter family of B\"acklund transformation and the existence of infinitely many conservation laws for the sine-Gordon equation. It is thus a key ingredient in the solution of the sine-Gordon equation by the method of inverse scattering. More generally one may consider the general class of partial differential equations describing pseudo-spherical surfaces with the special property that one of the components $f_{ij}$ can be chosen to be a continuous parameter. Such equations are said to describe {\it{$\eta$  pseudo-spherical surfaces}}. The evolution equations (\ref{evol}) describing $\eta$ pseudo-spherical surfaces have been completely classified by Chern and Tenenblat in~\cite{ChernTenenblat}, whenever $f_{21}=\eta$, $F$ and $f_{ij}$ depend on $u$ and finitely many derivatives of $u$ with respect to $t$ and $x$. Any differential equation describing $\eta$ pseudo-spherical surfaces is the integrability condition of a linear system of the form 
\begin{equation*}\label{linear}
dv^{1}=\frac{1}{2}\big(\omega_{2}\,v^{1}+(\omega_{1}-\omega_{3})\,v^{2}\big),\quad dv^{2}=\frac{1}{2} \big((\omega_{1}+\omega_{3})\,v^{1}-\omega_{2}\,v^{2}\big),
\end{equation*}
which may be used to solve the given differential equation by the method of inverse scattering~\cite{BealsRabeloTenenblat}, with $\eta$ playing the role of the spectral parameter for the scattering problem. It is also shown in~\cite{CavalcanteTenenblat} that one can generate infinite sequences of conservation laws for the class of differential equations describing $\eta$ pseudo-spherical surfaces by making use of the structure equations~(\ref{struct}), although some of these conservation laws may end up being non-local. Important further developments of these ideas around this theme can be found in 
~\cite{CastroTenenblat}, ~\cite{CatalanoOliveira16}, ~\cite{FT},
~\cite{FoursovOlverReyes}-\cite{JorgeTenenblat},
~ \cite{KamranTenenblat}-\cite{Reyes106}. 

We should also remark at this stage that given a differential equation describing pseudo-spherical surfaces, the choice of $1$-forms satisfying the structure equations (\ref{struct}) is generally not unique. For example the $1$-forms given by
\begin{equation}\label{sGcof2}
\omega_1 = \cos \frac{u}{2}( dx+dt),\quad \omega_2 =\sin \frac{u}{2} (dx - dt),\quad \omega_3 = \frac{u_{x}}{2} dx - \frac{u_{t}}{2} dt,
 \end{equation}
which are different from the $1$-forms given in (\ref{sGcof1}), will also satisfy the structure equations (\ref{struct}) whenever $u$ is a solution of the sine-Gordon equation (\ref{sG}). 

Starting with~\cite{KKT1}, we have initiated the study of differential equations describing pseudo-spherical surfaces from an extrinsic perspective, in which we focus on the properties of the local isometric immersions of the metrics (\ref{metric}) associated to the solutions of the equations. It is indeed a classical result that any metric (\ref{metric}) of constant negative scalar curvature can be locally isometrically immersed in ${\mathbb{E}}^{3}$. For the metrics defined by solutions $u: U\subset \mathbb{R}^{2} \to \mathbb{R} $ of equations describing pseudo-spherical surfaces, it is thus natural to ask in view of the integrability properties enjoyed by this class of equations if the second fundamental form of the immersion can be expressed in a simple way in terms of the solution $u$. This turns out to be effectively the case for the sine-Gordon equation (\ref{sG}). Indeed let us first recall the components $a,b,c$ of the second fundamental form of a local isometric immersion of a pseudo-spherical surface into $\mathbb{E}^{3}$ are defined by the relations
\begin{equation*}\label{omega13}
\omega_{13} = a\,\omega_1+b\,\omega_2, \quad \omega_{23} = b\,\omega_1+c\,\omega_2,
\end{equation*}
where the $1$-forms $\omega_{13}, \omega_{23}$ satisfy the structure equations
\begin{equation*}\label{Codazzi}
d\omega_{13} = \omega_{12}\wedge\omega_{23}, \quad d\omega_{23} = \omega_{21}\wedge\omega_{13},
\end{equation*}
equivalent to the Codazzi equations, and the Gauss equation, given by
\begin{equation*}\label{gauss}
ac-b^2=-1,
\end{equation*}
for a pseudo-spherical surface. For the sine-Gordon equation, with the choice of $1$-forms $\omega_{1}, \omega_{2}$ and $\omega_{3}=\omega_{12}$ given by (\ref{sGcof2}), the $1$-forms $\omega_{13},\omega_{23}$ are easily computed to be
\begin{equation*}
\omega_{13} =\sin\frac{u}{2} (dx+dt) =  \tan \frac{u}{2}\omega_1,\quad 
\omega_{23} =-\cos\frac{u}{2} (dx - dt) = -\cot \frac{u}{2}\omega_2.
\end{equation*}
We thus observe the remarkable property that the components $a,b,c$ of the second fundamental form depend only on $u$ through some simple, closed-form expressions. It would therefore not be unreasonable to expect a similar property to hold for all equations describing pseudo-spherical surfaces, where the requirement could be relaxed by allowing the coefficients $a,b.c$ of the second fundamental form to depend on $u$ and {\emph{finitely many of its derivatives}. In ~\cite{KKT1,KKT2}, we began to investigate the class of differential equations describing pseudo-spherical surfaces from this extrinsic perspective. Thus in ~\cite{KKT1}, we proved that for second-order equations of the form 
\begin{equation*}
\frac{\partial u}{\partial t}=F(u,\frac{\partial u}{\partial x},\frac{\partial^{2} u}{\partial x^{2}}),
\end{equation*}
and 
\begin{equation*}
\frac{\partial^{2}u}{\partial x\partial t}=F(u,\frac{\partial u}{\partial x}),
\end{equation*}
describing $\eta$ pseudo-spherical surfaces, the only equations for which $a,b,c$ will depend on $u$ and finitely many derivatives of $u$ are given by the sine-Gordon equation (\ref{sG}), and evolution equations of the form 
\begin{equation}\label{secondevol}
\frac{\partial u}{\partial t}=\frac{1}{f_{11,u}}\big(f_{12,\frac{\partial u}{\partial x}}\frac{\partial^{2}u}{\partial x^{2}}+f_{12,u}\frac{\partial u}{\partial x}\mp(\beta f_{11}-\eta f_{12})\big),
\end{equation}
where $f_{11,u}\neq 0$ and $f_{12,\frac{\partial u}{\partial x}}\neq 0$, where in the latter case the components $a,b,c$ of the second fundamental form are {\emph{universal}} functions of $x,t$, independent of $u$. Results of a similar nature were obtained in \cite{CastroKamran} for third-order equations of the form
\begin{equation*}
\frac{\partial u}{\partial t}-\frac{\partial^{3}u}{\partial x^{2}\partial t}=\lambda u \frac{\partial^{3} u}{\partial x^{3}}+G(u,\frac{\partial u}{\partial x},\frac{\partial^{2}u}{\partial x^{2}}).
\end{equation*}
and in ~\cite{CatalanoOliveira17} for a class of second order evolution equations of type 
$$\frac{\partial u}{\partial t}=A(x,t,u)\frac{\partial^{2}u}{\partial x^{2}}+B(x,t,u,\frac{\partial u}{\partial x})
$$ and for $k$-th order 
evolution equations in conservation law form. 

In~\cite{KKT2}, the same question was considered for the evolution equations (\ref{evol}) of order $k\geq 3$ classified in ~\cite{ChernTenenblat}, 
 where we proved as a first result that the $a,b,c$ are again necessarily {\emph{universal functions}} of $x,t$, independent of $u$. 
Our purpose in the present paper is to complete this analysis by determining the analogue of the form (\ref{secondevol}) for $k$-th order evolution equations (\ref{evol}). We now state our main result:

\begin{Thm}\label{MainRes} 
Except for $k$-th order evolution equations  of the form 
\begin{equation}\label{EqException}
\dfrac{\partial u}{\partial t} = \frac{1}{f_{11,u}}\big(\sum_{i=0}^{k-1} f_{12,\partial^i u/\partial x^i} \cdot \dfrac{\partial^{i+1} u}{\partial x^{i+1}} \mp (\beta f_{11} - \eta f_{12})\big), \qquad k\geq 2, 
\end{equation}
where  $f_{11, u} \neq 0$ and $f_{12,\frac{\partial^{k-1} u}{\partial x^{k-1}}}\neq 0$, there exists no $k$-th order  evolution equation of order $k\geq 2$ describing $\eta$ pseudo-spherical surfaces, with $1$-forms  (\ref{forms}) given as in \cite{ChernTenenblat}, with the property that the coefficients of the second fundamental forms of the local isometric immersions of the surfaces associated to the solutions $u$ of the equation depend on a jet of finite order of $u$.  Moreover, the coefficients of the second fundamental forms of the local isometric immersions of the surfaces determined by the solutions $u$ of (\ref{EqException}) are universal, i.e., they are universal functions of  $\eta x+\beta t$, independent of $u$. 
 \end{Thm}

This theorem provides further evidence the special place that the sine-Gordon equation appears to occupy amongst all integrable equations from the perspective provided by the theory of differential equations describing $\eta$ pseudo-spherical surfaces. 

We point out that the universal coefficients of the second fundamental forms of the isometric immersions mentioned in  Theorem \ref{MainRes} are explicitly given in Proposition \ref{Prop1}. We now prove a consequence of our main result. 

\begin{Cor}  For each solution $u$ of an equation of type (\ref{EqException}), there exists a foliation of the domain of $u$ by straight lines with the property that when the metric of constant negative Gaussian curvature $K=-1$ associated to $u$ through~(\ref{metric}) is locally isometrically immersed as a surface $S \subset {\mathbb{E}}^{3}$, the mean curvature of $S$ is constant along the curves defined by the images under the immersion of the lines of this foliation.
\end{Cor}
\begin{proof}
For each solution $u$ of any equation of type (\ref{EqException}), the associated 1-forms (\ref{forms}) define a metric with Gaussian  curvature 
$K=-1$. It follows from Theorem \ref{MainRes} that a local isometric immersion of such a metric into $\mathbb{R}^3$ is determined by 
the coefficients of the second fundamental form, which are functions of $\eta x+\beta t$. Now consider $(x,t)\in \mathbb{R}^2$ such that the 
straight line $\eta x+\beta t=\delta$, $\delta\in \mathbb{R}$ is contained in the domain of definition of the immersion. The domain is foliated by such straight lines.  For each  $\delta$, the image of the straight line is a curve in the surface. Along this curve 
the coefficients of the second fundamental form are constants determined by 
$\delta$. Since the mean curvature $H$ of the surface is given by the trace 
of the second fundamental form, it follows that $H$ is constant along any such curve. 
\end{proof}

Before proving Theorem \ref{MainRes}, we observe that a similar result on the mean curvature of the immersed surface also holds for  the main results obtained in \cite{KKT1} and \cite{CastroKamran}. In fact,  
the  arguments are the same as those used  
in the proof of the Corollary above.

\section{Proof of Theorem \ref{MainRes}}

The proof of Theorem \ref{MainRes} is based on an order analysis of the Codazzi and Gauss equations that govern the local isometric immersions of pseudo-spherical surfaces in $\mathbb{E}^{3}$, considering in turn each branch of the Chern-Tenenblat classification of $k$-th order evolution equations describing $\eta$ pseudo-spherical surfaces~\cite{ChernTenenblat}. In order to carry out this analysis, one should first express these equations in terms of the components $f_{ij}$ of the $1$-forms that appear in the formulation of the problem.  These conditions have already been worked out in \cite{KKT1,KKT2}; the Codazzi equations read

\begin{eqnarray}\label{Eq1}
 f_{11} D_{t}a + f_{21}D_{t}b - f_{12}D_{x}a - f_{22}D_{x}b - 2b \Delta_{13}+ (a-c)\Delta_{23} = 0,\\\label{Eq2}
 f_{11} D_{t}b + f_{21}D_{t}c - f_{12}D_{x}b - f_{22}D_{x}c +(a-c) \Delta_{13}+ 2b\Delta_{23} = 0,
 \end{eqnarray}
 where 
$$\Delta_{12} := f_{11}f_{22} - f_{21} f_{12}; \quad \Delta_{13} := f_{11}f_{32} - f_{31} f_{12}; \quad \Delta_{23} := f_{21}f_{32} - f_{31} f_{22},$$
and where the operators $D_{t}$ and $D_{x}$ are total derivative operators, while the Gauss equation is given by 
\begin{equation}\label{Gauss}
ac-b^2=-1.
\end{equation}

In the case of a differential equation describing $\eta$ pseudo-spherical surfaces (with $f_{21}=\eta$), the structure equations (\ref{struct}) are equivalent to 
\begin{eqnarray*} \label{SEq1}
D_t f_{11} - D_x f_{12} = \Delta_{23},\\\label{SEq2}
D_x f_{22} = \Delta_{13},\\\label{SEq3}
D_t f_{31} - D_x f_{32} = -\Delta_{12},
\end{eqnarray*}
where $D_t$ and $D_x$ are the total derivative operators and 
\begin{equation*}
\Delta_{12} := f_{11}f_{22} -\eta f_{12}\neq 0, \quad  \Delta_{13} := f_{11}f_{32} - f_{31}f_{12}, \quad  \Delta_{23} = \eta f_{32} - f_{31}f_{22}.
\end{equation*}
We shall use the notation 
\begin{equation*}
z_{i}=u_{x^{i}}=\frac{\partial^{i}u}{\partial x^{i}},\quad 0\leq i \leq k,
\end{equation*}
introduced in \cite{ChernTenenblat} to denote the derivatives of $u$ with respect to $x$ and write the evolution equation (\ref{evol}) as 
\begin{equation*}\label{evolkz}
 z_{0,t} = F(z_0, z_1, \dots, z_k).
\end{equation*}
We will thus think of $(t,x,z_{0},\dots,z_{k})$ as local coordinates on an open set of the submanifold of the jet space $J^{k}({\mathbb R}^{2},{\mathbb R})$ defined by the differential equation (\ref{evol}). We will use the following lemma from~\cite{ChernTenenblat} which expresses the necessary and sufficient conditions for the structure equations (\ref{struct}) to hold:

\begin{Lem} {\bf \cite{ChernTenenblat}}  Let (\ref{evol}) be a $k$-th order evolution equation describing $\eta$ pseudo-spherical surfaces, with associated $1$-forms (\ref{forms}) such that $f_{21}=\eta$. Then necessary and sufficient conditions for the structure equations (\ref{struct}) to hold are given by
\begin{eqnarray} \label{std1}
f_{11,z_k} = \cdots = f_{11,z_1} = 0, \quad
f_{21} = \eta, \quad 
f_{31,z_k} = \cdots = f_{31,z_1} = 0,\\ \label{std2}
f_{12,z_{k}} = 0, \quad 
f_{22, z_k} = f_{22,z_{k-1}} = 0, \quad 
f_{32,z_{k}} = 0,  
\end{eqnarray}
\begin{equation}\label{std7}
f_{11,z_0}^2 + f_{31,z_0}^2 \neq 0,  
\end{equation}
\begin{eqnarray}\label{SEq1n}
f_{11,z_0} F = \sum_{i=0}^{k-1} f_{12,z_i}z_{i+1} + \eta f_{32} - f_{31} f_{22},\\\label{SEq2n}
\sum_{i=0}^{k-2} f_{22,z_i}z_{i+1} = f_{11}f_{32} - f_{31}f_{12}, 
\\\label{SEq3n}
f_{31,z_0} F = \sum_{i=0}^{k-1} f_{32,z_i}z_{i+1} + \eta f_{12} - f_{11} f_{22}, 
\end{eqnarray}
and 
\begin{equation}\label{Delta12}
f_{11}f_{22}-\eta f_{12}\neq 0.
\end{equation}
\end{Lem}

\vspace{.1in}

As stated in the Introduction, Theorem \ref{MainRes} states that the $k$-th order evolution equation of type (\ref{evol}) describing $\eta$ pseudo-spherical surfaces are divided into two categories when viewed from the perspective of the local isometric immersions of pseudo-spherical metrics defined by their solutions: either the coefficients of the second fundamental forms are universal functions of $x$ and $t$, independent $u$, or the coefficients of the second fundamental form depend of a jet of infinite order of $u$. In order to prove our theorem,  we shall make use of the classification results for $k$-th order evolution equations describing $\eta$ pseudo-spherical surfaces given in Theorems 2.2, 2.3, 2.4 and 2.5 in ~\cite{ChernTenenblat}. These theorems considered five groups of equations, 
summarized and reorganized below,   
according to the properties of the following functions first introduced 
in ~\cite{ChernTenenblat}
\begin{eqnarray*}\label{LHPM}
\begin{array}{llllll}H &=& f_{11}f_{11,z_{0}} - f_{31}f_{31,z_{0}},& L &=& f_{11}f_{31,z_{0}} - f_{31}f_{11,z_{0}}, 
\end{array}
\end{eqnarray*}
{\bf Remark 1.}
\quad A $k$-th order evolution equation $z_{0,t}=F(z_0,...,z_k)$   describing  $\eta$ pseudo-spherical surfaces with associated 1-forms 
$\omega_i=f_{i1}dx+f_{i2}dt, 1\leq i\leq 3$, where $f_{ij}, 1\leq j\leq 2$ 
satisfy (\ref{std1})-(\ref{Delta12}), is in one of the following five groups:
\begin{itemize}
\item [{\bf I:}] $L=0$ with $f_{31}=\lambda f_{11}\neq 0$, $\lambda^2-1=0$. In this case, $f_{22}$ does not depend on $z_i$, $0\leq i\leq k$ and $f_{32}=\lambda f_{12}$.
\item [{\bf II:}] $L=0$ with $f_{31}=\lambda f_{11}\neq 0$, $\lambda^2-1\neq 0$. In this case, $f_{22,z_{k-2}}=0$.
\item [{\bf III:}] $L=0$ and $H\neq 0$, i.e., $f_{11}=0$ and $f_{31,z_0}\neq 0$ or 
$f_{31}=0$ and $f_{11,z_0}\neq 0$.
\item [{\bf IV:}] $L\neq 0$ and $H=0$, i.e., $f_{31}^2-f_{11}^2=C\neq 0$.
\item [{\bf V:}]  $HL\neq 0$.
\end{itemize}
We observe that equations of Groups I and II were treated in Theorem 2.4 (a) and (b) in 
~\cite{ChernTenenblat} respectively and  equations of Groups III, IV and V were treated in Theorems 2.3, 2.5 and 2.2 in ~\cite{ChernTenenblat} respectively.

If the coefficients $a, b, c$ of the 1-forms $\omega_{13}$ and $\omega_{23}$ depend of a jet of finite order of $u=z_{0}$, that is  $a, b, c$ are functions of $x, t, z_0, \dots, z_l$ for some finite $l$, then (\ref{Eq1}) and (\ref{Eq2}) become
\begin{eqnarray*}\label{Eq1D}
f_{11}a_t + \eta b_t - f_{12}a_x - f_{22} b_x - 2b \Delta_{13} + (a-c)\Delta_{23}  - \sum_{i=0}^l (f_{12}a_{z_i} + f_{22}b_{z_i})z_{i+1} \\
+\sum_{i=0}^l (f_{11}a_{z_i} + \eta b_{z_i})z_{i,t} = 0,
\end{eqnarray*}
and
\begin{eqnarray*}\label{Eq2D}
f_{11}b_t + \eta c_t - f_{12}b_x - f_{22} c_x + (a-c) \Delta_{13} + 2b\Delta_{23}  - \sum_{i=0}^l (f_{12}b_{z_i} + f_{22}c_{z_i})z_{i+1} \\
+\sum_{i=0}^l  (f_{11}b_{z_i} + \eta c_{z_i})z_{i,t} = 0.
\end{eqnarray*}

In~\cite{KKT2}, we showed that if the coefficients of the second fundamental form $a, b$ and $c$ depend on a jet of finite order of $u=z_{0}$,  then $a, b$ and $c$ are universal, that is $l=0$ and $a, b$ and $c$ depend at most on $x$ and $t$ only. Therefore, equations (\ref{Eq1}) and (\ref{Eq2}) become
\begin{eqnarray}\label{Eq1-2n}
f_{11}a_t + \eta b_t - f_{12}a_x - f_{22} b_x - 2b (f_{11}f_{32}-f_{31}f_{12}) + (a-c)(\eta f_{32}-f_{31}f_{22})  = 0,\\\label{Eq2-2n}
f_{11}b_t + \eta c_t - f_{12}b_x - f_{22} c_x + (a-c)(f_{11}f_{32}-f_{31}f_{12})  + 2b  (\eta f_{32}-f_{31}f_{22}) = 0,  
\end{eqnarray}
where $a_{x}, b_{x}, c_{x}$ are the partial derivatives of $a, b, c$ with respect to $x$, and $a_{t}, b_{t}, c_{t}$ are the partial derivatives of $a, b, c$ with respect to $t$. 

We now consider in turn all the cases listed in Remark 1 and state the conclusion for each case in the form of a proposition. The first proposition pertains to evolution equations in item I  of Remark 1.

\begin{Prop} \label{Prop1} Let 
\begin{equation}\label{EvolEx1}
\dfrac{\partial u}{\partial t} = \frac{1}{f_{11,u}}\big(\sum_{i=0}^{k-1} f_{12,\partial^i u/\partial x^i} \cdot \dfrac{\partial^{i+1} u}{\partial x^{i+1}} \mp ( \beta f_{11} - \eta f_{12})\big), 
\end{equation}
where $f_{11, u} \neq 0$ and $f_{12,\frac{\partial^{k-1} u}{\partial x^{k-1}}}\neq 0$, be a $k$-th order evolution equation, $k\geq 2$, which describes $\eta$ pseudo-spherical surfaces with 1-forms $\omega_i$ as in \cite{ChernTenenblat}. There exists a local isometric immersion in $\mathbb{R}^3$  of a pseudo-spherical surface, determined by a 
solution $u$, for which the coefficients $a,b,c$ of the second fundamental form depend on a jet of finite order of $u$ if, and only if, the coefficients are  universal and are given by 
\begin{eqnarray}\label{auniversal}
a&=& \sqrt{le^{\pm2(\eta x + \beta t )} - \gamma^2e^{\pm4(\eta x + \beta t )} - 1}, \\\label{buniversal}
b&=& \gamma e^{\pm2(\eta x + \beta t)}, \\\label{cuniversal}
c&=& \dfrac{\gamma^2e^{\pm4(\eta x + \beta t)} - 1}{ \sqrt{le^{\pm2(\eta x + \beta t )} - \gamma^2e^{\pm 4(\eta x + \beta t )} - 1}}, 
\end{eqnarray}
$l,\gamma \in \mathbb{R}, \, l>0$   and   $l^2 >4\gamma^2$.  The 1-forms are defined on a strip of $\mathbb{R}$ where 
\begin{equation}\label{strip}
\log\sqrt{\dfrac{l-\sqrt{l^2-4\gamma^2}}{2\gamma^2}}< \pm( \eta x+\beta t)< \log\sqrt{\dfrac{l+\sqrt{l^2-4\gamma^2}}{2\gamma^2}}.   
 \end{equation}
 Moreover, the constants $l$ and $\gamma$ have to be chosen so that the strip intersects the domain of the solution of the evolution equation. 
\end{Prop}

\begin{proof} For evolution equations of type  I, we have 
$f_{31}=\pm f_{11}$, $f_{32}=\pm f_{12}$, 
 $f_{11,z_0}\neq 0$, $f_{12,z_{k-1}}\neq 0$ and the equation is given by 
 (\ref{EvolEx1}). Moreover,  $f_{22}=\beta $ is independent of $z_0, \dots, z_k$.  Equations (\ref{Eq1-2n}) and (\ref{Eq2-2n}) become
\begin{eqnarray}\label{eq1S}
f_{11}a_t + \eta b_t  - f_{12}a_x - \beta b_x \mp (a-c)(\beta f_{11} - \eta f_{12}) = 0,\\ \label{eq2S}
f_{11}b_t + \eta c_t -  f_{12}b_x - \beta c_x \mp 2b(\beta f_{11} - \eta f_{12})= 0.
\end{eqnarray}
Since  $k\geq 2$, differentiating the latter two equations with respect to $z_{k-1}$ leads to 
\begin{eqnarray*}\label{eqA1}
-f_{12,z_{k-1}} a_x \pm (a-c)\eta f_{12,z_{k-1}} = 0,\\\label{eqA2}
-f_{12,z_{k-1}} b_x \pm 2b \eta f_{12,z_{k-1}} = 0.
\end{eqnarray*}
Note that $f_{12,z_{k-1}} \neq 0$ by hypothesis, which means that these equations  simplify to 
\begin{eqnarray}\label{Cax=0}
a_{x} \mp \eta (a-c) = 0,\\\label{Cbx=0}
b_{x}\mp  2\eta b = 0.
\end{eqnarray}
Taking into account (\ref{Cax=0}) and (\ref{Cbx=0}), equations (\ref{eq1S}) and  (\ref{eq2S}) become 
\begin{eqnarray}\label{eq1S2}
f_{11}a_t + \eta b_t  - \beta b_x \mp  (a-c) \beta f_{11} = 0,\\ \label{eq2S2}
f_{11}b_t + \eta c_t - \beta c_x \mp 2b \beta f_{11} = 0.
\end{eqnarray}
Differentiating (\ref{eq1S2}) and (\ref{eq2S2}) with respect to $z_{0}$, with $f_{11,z_0} \neq 0$, leads to
\begin{eqnarray}\label{Cat=0}
a_{t} \mp \beta (a-c) = 0,\\\label{Cbt=0}
b_{t}\mp  2\beta b = 0,
\end{eqnarray}
and hence, (\ref{eq1S2}) and (\ref{eq2S2}) become
\begin{eqnarray}\label{btbx=0}
\eta b_{t} - \beta b_{x} = 0,\\\label{ctcx=0}
\eta c_{t} - \beta c_{x} = 0.
\end{eqnarray}
Note that  (\ref{Cax=0}) and (\ref{Cat=0}) imply 
\begin{eqnarray*}\label{atbx=0}
\eta a_{t} - \beta a_{x} = 0.
\end{eqnarray*}
 From (\ref{Cbx=0}) and (\ref{Cbt=0}), we conclude that 
\begin{equation}\label{b}
b= \gamma e^{\pm 2(\eta x + \beta t)}, \quad \gamma \in \mathbb{R}.
\end{equation}
Note that $a\neq 0$. Otherwise, if $a=0$, then (\ref{Cax=0}) implies that $c=0$ and the Gauss equation leads to $b=\pm1$ which contradicts (\ref{Cbx=0}). Therefore, from the Gauss equation  we have 
\begin{equation}\label{cGauss}
c={(b^2-1)}{a^{-1}}.
\end{equation}
 Then, in view of (\ref{b}),  
equations (\ref{Cax=0}) and (\ref{Cat=0}) reduce to 
\begin{eqnarray*}
aa_x\mp \eta (a^2 - \gamma^2e^{\pm 4(\eta x + \beta t)} + 1) = 0,\\
aa_t \mp\beta( a^2  -  \gamma^2e^{\pm 4(\eta x + \beta t)} + 1) = 0.\\
\end{eqnarray*}
The latter system leads then to 
\begin{equation*}
a= \sqrt{le^{\pm 2(\eta x + \beta t )} - \gamma^2e^{\pm 4(\eta x + \beta t )} - 1}, \quad l\in \mathbb{R},
\end{equation*}
which is defined wherever $le^{\pm 2(\eta x + \beta t )} - \gamma^2e^{\pm  4(\eta x + \beta t )} - 1>0$. 
Hence $l>0$ and 
\[ 
\dfrac{l-\sqrt{l^2-4\gamma^2}}{2\gamma^2} < e^{\pm 2(\eta x + \beta t )}< \dfrac{l+\sqrt{l^2-4\gamma^2}}{2\gamma^2},
\]
{\it i.e.}, $a$ is defined on the strip described by (\ref{strip}).  
Now,  from (\ref{cGauss}), we obtain 
\begin{equation*}
c= \dfrac{\gamma^2e^{\pm 4(\eta x + \beta t)} - 1}
{ \sqrt{le^{\pm 2(\eta x + \beta t )} - \gamma^2e^{\pm 4(\eta x +\beta t )} - 1}}. 
\end{equation*}
 A straightforward computation shows that the converse holds. 
Finally, we observe that given a solution of the evolution equation, in order to have an immersion, one has to choose the constants $l$ and $\gamma$, such that the strip (\ref{strip}) intersects the domain
of the solution in $\mathbb{R}^2$.
\end{proof}


Next, we consider the evolution  equations covered in item II of Remark 1.  

\begin{Prop}\label{Prop2}   For $k$-th evolution equations, $k\geq 2$,  describing $\eta$ pseudo-spherical surfaces with  associated $1$-forms $\omega^{i} = f_{i1}dx + f_{i2}dt$, $1\leq i\leq 3$, where $f_{31} = \lambda f_{11} \neq 0$, and $\lambda^2 \neq 1$, the system of equations (\ref{Eq1}), (\ref{Eq2}) and  (\ref{Gauss}) is inconsistent. 
\end{Prop}

\begin{proof}When $f_{31} = \lambda f_{11} \neq 0$, with $\lambda^2 \neq 1$, the structure equations (\ref{SEq1n}), (\ref{SEq2n})  and (\ref{SEq3n}) can be rewritten as
\begin{eqnarray}\label{seq1-B}
f_{11,z_0}F = \sum_{i=0}^{k-1}f_{12,z_{i}}z_{i+1} + \eta f_{32} - \lambda f_{11}f_{22}, \\\label{seq2-B}
\sum_{i=0}^{k-2} f_{22,z_i}z_{i+1} = f_{11}(f_{32} - \lambda f_{12}),\\\label{seq3-B}
\lambda f_{11,z_0}F = \sum_{i=0}^{k-1}f_{32,z_{i}}z_{i+1} + \eta f_{12} -  f_{11}f_{22} = 0. 
\end{eqnarray}
Differentiating (\ref{seq1-B}) with respect to $z_{k}$ leads to $f_{11,z_0}F_{z_k} = f_{12,z_{k-1}}$. Because $f_{11,z_0} \neq 0$ (otherwise (\ref{std7})
would fail) and $F_{z_{k}} \neq 0$ (otherwise $F$ is not a $k$-th order evolution equation), we conclude that $f_{12,z_{k-1}} \neq 0$.  Subtracting $\lambda$ times equation (\ref{seq1-B}) from (\ref{seq3-B}) leads to 
\begin{equation}\label{seq3-B1}
\sum_{i=0}^{k-1} (f_{32,z_i} - \lambda f_{12,z_i}) z_{i+1} + \eta (f_{12} - \lambda f_{32}) + (\lambda^2 - 1)f_{11}f_{22} = 0. 
\end{equation}
Differentiating the latter with respect to $z_{k}$ leads to 
\begin{equation}\label{condB1}
f_{32,z_{k-1}} - \lambda f_{12,z_{k-1}} = 0. 
\end{equation}

Since  $k\geq 2$, differentiating equation (\ref{seq2-B}) with respect to $z_{k-1}$ and taking into account (\ref{condB1}) leads to 
\begin{equation}\label{condf22neq0}
f_{22,z_{k-2}} = 0.
\end{equation}
Differentiating (\ref{seq3-B1}) with respect to $z_{k-1}$ leads to 
\begin{equation}\label{condB2}
f_{32,z_{k-2}} - \lambda f_{12,z_{k-2}} + \eta (f_{12,z_{k-1}} - \lambda f_{32, z_{k-1}}) = 0. 
\end{equation}
Substituting $f_{32,z_{k-1}}$ by $\lambda f_{12,z_{k-1}}$, which is just equation (\ref{condB1}), equation (\ref{condB2}) becomes 
\begin{equation*}\label{condB3}
f_{32,z_{k-2}} - \lambda f_{12,z_{k-2}}  =  \eta (\lambda^2 - 1)f_{12,z_{k-1}}. 
\end{equation*}
Because neither of $\eta$, $\lambda^2 -1$ and $f_{12,z_{k-1}}$ is zero, we conclude that when $f_{31} = \lambda f_{11} \neq 0$  and $\lambda^2 \neq 1$, we have 
\begin{equation}\label{condB4}
f_{32,z_{k-2}} - \lambda f_{12,z_{k-2}}  \neq 0. 
\end{equation}

Consider now equations (\ref{Eq1-2n}) and (\ref{Eq2-2n}), which can be written as follows:
\begin{eqnarray}\label{eq1B1}
f_{11}a_t + \eta b_t - f_{12}a_x - f_{22}b_x - 2bf_{11}(f_{32} - \lambda f_{12}) + (a-c)(\eta f_{32} - \lambda f_{11}f_{22}) = 0,\\\label{eq2B1}
f_{11}b_t + \eta c_t - f_{12}b_x - f_{22}c_x +(a-c)f_{11}(f_{32} - \lambda f_{12}) + 2b(\eta f_{32} - \lambda f_{11}f_{22}) = 0. 
\end{eqnarray}
 Differentiating (\ref{eq1B1}) and (\ref{eq2B1}) with respect to $z_{k-1}$ and taking into account (\ref{condf22neq0}) leads to 
\begin{eqnarray*}\label{eq1B2}
-f_{12,z_{k-1}} a_x - 2b f_{11} (f_{32,z_{k-1}} - \lambda f_{12,z_{k-1}}) + (a-c)\eta f_{32,z_{k-1}} = 0,\\\label{eq2B2}
-f_{12,z_{k-1}} b_x - 2b f_{11} (f_{32,z_{k-1}} - \lambda f_{12,z_{k-1}}) + 2b\eta f_{32,z_{k-1}} = 0. 
\end{eqnarray*}
Taking into account (\ref{condB1}) and the fact that $f_{12,z_{k-1}}\neq 0$, these two equations reduce to 
\begin{eqnarray*}\label{eq1B3}
a_x  =  (a-c)\eta \lambda,\\\label{eq2B3}
b_x = 2b \eta \lambda.
\end{eqnarray*}
Substituting the latter two equations in (\ref{eq1B1}) and (\ref{eq2B1}) leads 
\begin{eqnarray*}\label{eq1B4}
f_{11}a_t + \eta b_t - f_{12}\eta \lambda (a-c) - f_{22} \eta \lambda 2b - 2bf_{11}(f_{32} - \lambda f_{12}) + (a-c)(\eta f_{32} - \lambda f_{11}f_{22}) = 0,\\\label{eq2B4}
f_{11}b_t + \eta c_t - f_{12} \eta \lambda 2b  - f_{22}c_x +(a-c)f_{11}(f_{32} - \lambda f_{12}) + 2b(\eta f_{32} - \lambda f_{11}f_{22}) = 0. 
\end{eqnarray*}

\vspace{.1in}

If $k\geq 3$, differentiating these two equations  with respect to $z_{k-2}$ leads to 
\begin{eqnarray*}
-f_{12,z_{k-2}} \eta \lambda (a-c) - 2b f_{11} (f_{32,z_{k-2}} - \lambda f_{12,z_{k-2}}) + (a-c) \eta f_{32,z_{k-2}} = 0,\\
-f_{12,z_{k-2}} \eta \lambda 2 b+  (a-c) f_{11} (f_{32,z_{k-2}} - \lambda f_{12,z_{k-2}}) + 2b\eta f_{32,z_{k-2}} = 0,
\end{eqnarray*}
which can be rewritten as
\begin{eqnarray*}\label{eq1B5}
(f_{32,z_{k-2}} - \lambda f_{12,z_{k-2}})((a-c)\eta-2bf_{11})=0,\\
(f_{32,z_{k-2}} - \lambda f_{12,z_{k-2}})((a-c)f_{11}+2b\eta).
\end{eqnarray*}
Since from (\ref{condB4}), $f_{32,z_{k-2}} - \lambda f_{12,z_{k-2}} \neq 0$, we can rewrite the two equations in matrix form as follows
\begin{equation*}
\left(\begin{array}{cc}\eta & -f_{11} \\f_{11} & \eta\end{array}\right)\left(\begin{array}{c}a-c \\2b\end{array}\right) = 0.
\end{equation*}
Since $f_{11,z_0}\neq 0$, the determinant  $f_{11}^2 + \eta^2 \neq 0$, and hence $a-c=b=0$. The latter contradicts the Gauss equation (\ref{Gauss}). 

When $k=2$, the proof of this proposition was given in \cite{KKT1}. 
Therefore,  we conclude that when $f_{31} = \lambda f_{11} \neq 0$, $\lambda^2 \neq 1$ for any  $k\geq 2$, the system of equations (\ref{Eq1}), (\ref{Eq2}) and  (\ref{Gauss}) is inconsistent. 
\end{proof}

The following Proposition is concerned with the evolution equations covered by item III of Remark 1.

\begin{Prop}\label{Prop3} For $k$-th evolution equations, $k\geq 2$,  describing $\eta$ pseudo-spherical surfaces with  associated $1$-forms $\omega^{i} = f_{i1}dx + f_{i2}dt$, $1\leq i\leq 3$, where either $f_{11} = 0$ or $f_{31} = 0$, the system of equations (\ref{Eq1}), (\ref{Eq2}) and  (\ref{Gauss}) is inconsistent. 
\end{Prop}

\begin{proof} 
i) When $f_{11} = 0$, 
the structure equations (\ref{SEq1n}), (\ref{SEq2n}) and (\ref{SEq3n}) can be rewritten as follows:
\begin{eqnarray}\label{se1C0}
\sum_{i=0}^{k-1} f_{12,z_i}z_{i+1} = -\eta f_{32} + f_{31}f_{22},\\\label{se2C0}
\sum_{i=0}^{k-2} f_{12,z_i}z_{i+1} = -f_{31}f_{12},\\\label{se3C0}
f_{31,z_0} F = \sum_{i=0}^{k-1} f_{32,z_i}z_{i+1} + \eta f_{12}.
\end{eqnarray}
Differentiating (\ref{se1C0}) with respect to $z_{k}$ leads to $f_{12,z_{k-1}} = 0$. 

Since  $k\geq 2$, differentiating (\ref{se2C0}) with respect to $z_{k-1}$,  leads to $f_{22,z_{k-2}} = 0$. Differentiating (\ref{se3C0}) with respect to $z_{k}$ leads to 
\begin{equation*}
f_{31,z_0} F_{z_{k}} = f_{32,z_{k-1}}. 
\end{equation*}
On one hand, $F_{z_{k}} \neq 0$ because the evolution equation $F$ is of order $k$. On the other hand, $f_{31,z_0} \neq 0$, otherwise (\ref{std7}) is not satisfied. We conclude then that
\begin{equation*}
f_{32,z_{k-1}}\neq 0. 
\end{equation*}
Equations (\ref{Eq1-2n}) and (\ref{Eq2-2n}) become
\begin{eqnarray*}\label{eq1C1}
\eta b_t - f_{12}a_x - f_{22}b_x + 2b f_{31}f_{12} + (a-c) (\eta f_{32} - f_{31}f_{22}) = 0\\\label{eq2C1}
\eta c_t - f_{12}b_x - f_{22}c_x - (a-c)f_{31}f_{12} + 2b (\eta f_{32} - f_{31}f_{22}) = 0
\end{eqnarray*}
Differentiating the latter two equations  with respect to $z_{k-1}$  leads to 
\begin{eqnarray*}
(a-c)\eta f_{32,z_{k-1}} = 0,\\
2b \eta f_{32,z_{k-1}} = 0,
\end{eqnarray*}
and since $f_{32,z_{k-1}}\neq 0$, we conclude that $a-c = b = 0$, which contradicts the Gauss equation (\ref{Gauss}).
  Therefore,  for any $k\geq 2$, when $f_{11} = 0$, the system of equations (\ref{Eq1}), (\ref{Eq2}) and  (\ref{Gauss}) is inconsistent.\\

ii) When $f_{31} = 0$, 
the structure equations (\ref{SEq1n}), (\ref{SEq2n}) and (\ref{SEq3n}) can be rewritten as follows:
\begin{eqnarray}\label{se1C}
f_{11,z_0} F = \sum_{i=0}^{k-1} f_{12,z_i}z_{i+1} + \eta f_{32}, \\\label{se2C}
\sum_{i=0}^{k-2} f_{12,z_i}z_{i+1} = f_{11}f_{32},\\\label{se3C}
\sum_{i=0}^{k-1} f_{32,z_i}z_{i+1} =f_{11}f_{22} - \eta f_{12}.
\end{eqnarray}
Differentiating (\ref{se3C}) with respect to $z_{k}$ leads to $f_{32,z_{k-1}} = 0$. 

Since  $k\geq 2$,  differentiating (\ref{se2C}) with respect to $z_{k-1}$ leads to $f_{22,z_{k-2}} = 0$. Differentiating (\ref{se1C}) with respect to $z_{k}$ leads to 
\begin{equation*}
f_{11,z_0} F_{z_{k}} = f_{12,z_{k-1}}.
\end{equation*}
On one hand, $F_{z_{k}} \neq 0$ because the evolution equation $F$ is of order $k$. On the other hand, $f_{11,z_0} \neq 0$, otherwise (\ref{std7}) is not satisfied. We conclude then that
\begin{equation}\label{f12zk-1neq0}
f_{12,z_{k-1}}\neq 0.
\end{equation}
Equations (\ref{Eq1-2n}) and (\ref{Eq2-2n}) become
\begin{eqnarray}\label{eq1D1}
f_{11}a_t + \eta b_t - f_{12} a_x - f_{22}b_x - 2bf_{11}f_{32} + (a-c)\eta f_{32} = 0,\\\label{eq2D1}
 f_{11}b_t + \eta c_t - f_{12} b_x - f_{22}c_x + (a-c)f_{11}f_{32} + 2b\eta f_{32} = 0.
\end{eqnarray}
Differentiating (\ref{eq1D1}) and (\ref{eq2D1}) with respect to $z_{k-1}$ and taking into account (\ref{f12zk-1neq0}) leads to 
\begin{equation*}
a_x = b_x = 0
\end{equation*}
With  (\ref{f12zk-1neq0}),  equations (\ref{eq1D1}) and (\ref{eq2D1}) become
\begin{eqnarray}\label{eq1D2}
f_{11}a_t + \eta b_t - 2bf_{11}f_{32} + (a-c)\eta f_{32} = 0,\\\label{eq2D2}
 f_{11}b_t + \eta c_t  - f_{22}c_x + (a-c)f_{11}f_{32} + 2b\eta f_{32} = 0.
\end{eqnarray}
Taking into account the fact that $f_{32, z_{k-1}} =0$, and then differentiating (\ref{se3C}) with respect to $z_{k-1}$ leads to 
\begin{equation*}
f_{32,z_{k-2}} = -\eta f_{12,z_{k-1}}
\end{equation*}
and because neither $\eta$ nor $f_{12,z_{k-1}}$ vanishes, we have
\begin{equation*}
f_{32,z_{k-2}} \neq 0.
\end{equation*}

If $k\geq 3$, differentiating then (\ref{eq1D2}) and (\ref{eq2D2}) with respect to $z_{k-2}$, then dividing by $f_{32,z_{k-2}}$, and rewriting the two equations in matrix form, leads to
\begin{equation*}
\left(\begin{array}{cc}\eta & -f_{11} \\f_{11} & \eta\end{array}\right)\left(\begin{array}{c}a-c \\2b\end{array}\right) = 0.
\end{equation*}
Since $f_{11,z_0}\neq 0$, the determinant  $f_{11}^2 + \eta^2 \neq 0$, and hence  $a-c=b=0$, which contradicts the Gauss equation (\ref{Gauss}). 

If $k=2$, the proof of this proposition was given in \cite{KKT1}.
Therefore, we conclude that  for any $k\geq 2$, when $f_{31} = 0$, the system of equations (\ref{Eq1}), (\ref{Eq2}) and  (\ref{Gauss}) is inconsistent. 
\end{proof}

Next, we consider the evolution equations in item IV of  Remark 1. 

\begin{Prop}\label{prop4} For $k$-th order evolution equations, $k\geq 2$,  describing $\eta$ pseudo-spherical surfaces with  associated $1$-forms $\omega^{i} = f_{i1}dx + f_{i2}dt$, $1\leq i\leq 3$, where $f_{31}^2 - f_{11}^2 = {C} \neq 0$, the system of equations (\ref{Eq1}), (\ref{Eq2}) and  (\ref{Gauss}) is inconsistent. 
\end{Prop}

\begin{proof} When $f_{31}^2 - f_{11}^2 =  C \neq 0$, then $H=0$ and $L\neq 0$. 
We consider 
the structure equations (\ref{SEq1n}), (\ref{SEq2n})  and (\ref{SEq3n}). 
On one hand, subtracting $f_{31}$ times (\ref{SEq3n}) from $f_{11}$ times (\ref{SEq1n}), and taking into account that $H=0$,  leads to 
\begin{equation*}
\sum_{i=0}^{k-1} (f_{11}f_{12,z_i} - f_{31}f_{32,z_i}) z_{i+1} + \eta (f_{11}f_{32} - f_{31}f_{12}) =HF= 0. 
\end{equation*}
Differentiating the latter with respect to $z_{k}$ leads to 
\begin{equation}\label{rel=0}
f_{11}f_{12,z_{k-1}} - f_{31}f_{32,z_{k-1}} = 0. 
\end{equation}
On the other hand,  subtracting $f_{31}$ times (\ref{SEq1n}) from $f_{11}$ times (\ref{SEq3n})  leads to 
\begin{equation*}
LF= \sum_{i=0}^{k-1} (f_{11}f_{32,z_i} - f_{31}f_{12,z_i}) z_{i+1} +  \eta(f_{11}f_{12} - f_{31}f_{32}) + Cf_{22} \neq 0.
\end{equation*}
Differentiating the latter with respect to $z_{k}$ leads to 
\begin{equation*}
LF_{z_{k}} = f_{11}f_{32,z_{k-1}} - f_{31}f_{12,z_{k-1}}. 
\end{equation*}
Note that because neither of $L$ and $F_{z_{k}}$ is zero, we have 
\begin{equation}\label{eqf32zk-1}
f_{11}f_{32,z_{k-1}} - f_{31}f_{12,z_{k-1}} \neq 0.
\end{equation}

Observe that since $f_{31}^2-f_{11}^2=C\neq 0$, we conclude that $f_{11}\neq 0$  and $f_{31}\neq 0$.  In fact, otherwise if $f_{11}=0$, then 
$f_{31}^2=C$, which contradicts (\ref{std7}). Similarly, one shows that 
$f_{31}\neq 0$.  
From  (\ref{rel=0}) and (\ref{eqf32zk-1}), we then conclude that  
  $f_{12,z_{k-1}} \neq 0$, and  $f_{32,z_{k-1}} \neq 0$.
Indeed, if $f_{12,z_{k-1}} = 0$ (resp. $f_{32,z_{k-1}} = 0$) then it follows from (\ref{rel=0}) that $f_{32,z_{k-1}}=0$ (resp. $f_{12,z_{k-1}} = 0$)  which  contradicts (\ref{eqf32zk-1}).  
 We conclude then from (\ref{rel=0}) that 
\begin{equation}\label{f1131frac}
\dfrac{f_{11}}{f_{31}} = \dfrac{f_{32,z_{k-1}}}{f_{12, z_{k-1}}}.
\end{equation}

In light of the above analysis, let's consider (\ref{Eq1-2n}) and (\ref{Eq2-2n}). 
Since $k\geq 2$, differentiating (\ref{Eq1-2n}) and (\ref{Eq2-2n}) with respect to $z_{k-1}$, and then dividing by $f_{12,z_{k-1}} \neq 0$, leads to 
\begin{eqnarray*}\label{eq1F1}
a_x = (a-c)\eta \dfrac{f_{32,z_{k-1}}}{f_{12,z_{k-1}}} - 2b \dfrac{f_{11}f_{32,z_{k-1}} - f_{31}f_{12,z_{k-1}}}{f_{12,z_{k-1}}}\\\label{eq2F1}
b_x = (a-c) \dfrac{f_{11}f_{32,z_{k-1}} - f_{31}f_{12,z_{k-1}}}{f_{12,z_{k-1}}} + 2b \eta \dfrac{f_{32,z_{k-1}}}{f_{12,z_{k-1}}}
\end{eqnarray*}
Taking into account  (\ref{f1131frac}) and $f_{31}^2 - f_{11}^2 = C$, these equations reduce to 
\begin{eqnarray*}\label{eq1F2}
a_x = (a-c)\eta\dfrac{f_{11}}{f_{31}} + 2b \dfrac{C}{f_{31}},\\\label{eq2F2*}
b_x = -(a-c)\dfrac{C}{f_{31}} + 2b \eta \dfrac{f_{11}}{f_{31}}. 
\end{eqnarray*}
Differentiating the latter two equations  with respect to $z_0$ leads to 
\begin{eqnarray*}
-\eta L (a-c) - C f_{31,z_0} 2b = 0,\\
C f_{31,z_0} (a-c) - \eta L 2b = 0,  
\end{eqnarray*}
which in matrix form become
\begin{equation*}
\left(\begin{array}{cc}-\eta L  & -Cf_{31,z_0} \\Cf_{31,z_0} & -\eta L\end{array}\right)\left(\begin{array}{c}a-c \\2b\end{array}\right) = 0. 
\end{equation*}
Because neither $\eta$, $C$, $f_{31,z_0}$ and $L$ is zero, the determinant $\eta^2 L ^2 + C^2 f_{31,z_0}^2 \neq 0$, and hence $a-c = b= 0$, which contradicts the Gauss equation (\ref{Gauss}). 
Therefore, we conclude that for any $k\geq 2$,  when $f_{31}^2 -f_{11}^2  = C\neq 0$, the system of equations (\ref{Eq1}), (\ref{Eq2}) and  (\ref{Gauss}) is inconsistent. 
\end{proof}

Finally, we have the following similar result for the evolution equations in item V of Remark 1. 

\begin{Prop}\label{prop5} For $k$-th order evolution equations, $k\geq 2$,  describing $\eta$ pseudo-spherical surfaces of type (\ref{evol}), with  associated $1$-forms $\omega^{i} = f_{i1}dx + f_{i2}dt$, $1\leq i\leq 3$, where $HL \neq 0$, the system of equations (\ref{Eq1}), (\ref{Eq2}) and  (\ref{Gauss}) is inconsistent. 
\end{Prop}

\begin{proof} We consider the  structure equations (\ref{SEq1n}), (\ref{SEq2n})  and (\ref{SEq3n} ). We are 
assuming that $HL \neq 0$, where $H= f_{11}f_{11,z_0} - f_{31}f_{31,z_0}$ and $L = f_{11}f_{31,z_0} - f_{31}f_{11,z_0}$, hence $f_{11}\neq 0$ and 
$f_{31}\neq 0$. Subtracting $f_{31}$ times (\ref{SEq3n}) from $f_{11}$ times (\ref{SEq1n}) leads to 
\begin{eqnarray}\label{HF}
HF = \sum_{i=0}^{k-1} (f_{11}f_{12,z_i} - f_{31}f_{32,z_i}) z_{i+1} + \eta (f_{11}f_{32} - f_{31}f_{12}), 
\end{eqnarray}
while subtracting $f_{31}$ times (\ref{SEq1n}) from $f_{11}$ times (\ref{SEq3n}) leads to 
\begin{eqnarray}\label{LF}
LF = \sum_{i=0}^{k-1} (f_{11}f_{32,z_i} - f_{31}f_{12,z_i}) z_{i+1} + (f_{31}^2 - f_{11}^2)f_{22}  + \eta (f_{11}f_{12} - f_{31}f_{32}). 
\end{eqnarray}
Differentiating (\ref{HF}) and (\ref{LF}) with respect to $z_{k}$ leads to 
\begin{eqnarray*}
f_{11}f_{12,z_{k-1}} - f_{31}f_{32,z_{k-1}} = HF_{z_k},\\
f_{11}f_{32,z_{k-1}} - f_{31}f_{12,z_{k-1}} = LF_{z_k},
\end{eqnarray*}
and since neither of $H$, $L$ and $F_{z_k}$ is zero, we have 
\begin{eqnarray}\label{neq1}
f_{11}f_{12,z_{k-1}} - f_{31}f_{32,z_{k-1}} \neq 0, \\\label{neq2}
f_{11}f_{32,z_{k-1}} - f_{31}f_{12,z_{k-1}} \neq 0.
\end{eqnarray}

Since $k\geq 2$, differentiating (\ref{SEq2n}) with respect to $z_{k-1}$ leads to 
\begin{equation}\label{101}
f_{11}f_{32,z_{k-1}} - f_{31}f_{12,z_{k-1}} = f_{22,z_{k-2}}, 
\end{equation}
and hence
\begin{equation}\label{102}
f_{22,z_{k-2}} \neq 0. 
\end{equation}
Note that when $HL\neq 0$, it follows from (\ref{HF}) and (\ref{LF})  that the expression of $F$ can be written in two equivalent ways: 
\begin{eqnarray*}\label{F1}
F = \sum_{i=0}^{k-1} \dfrac{f_{11}f_{12,z_i} - f_{31}f_{32,z_i}}{H} z_{i+1} + \eta \dfrac{f_{11}f_{32} - f_{31}f_{12}}{H},\\\label{F2}
F =  \sum_{i=0}^{k-1} \dfrac{f_{11}f_{32,z_i} - f_{31}f_{12,z_i}}{L} z_{i+1}  + \dfrac{(f_{31}^2 - f_{11}^2)f_{22}}{L} + \eta \dfrac{f_{11}f_{12} - f_{31}f_{32}}{L}. 
\end{eqnarray*}
Subtracting the last two equations leads then to 
\begin{equation}\label{eqFmF}
\begin{split}
\sum_{i=0}^{k-1} \bigg(\dfrac{f_{11}f_{12,z_i} - f_{31}f_{32,z_i}}{H} -  \dfrac{f_{11}f_{32,z_i} - f_{31}f_{12,z_i}}{L}  \bigg)z_{i+1} \\ + \eta \bigg(\dfrac{f_{11}f_{32}-f_{31}f_{12}}{H}  - \dfrac{f_{11}f_{12} - f_{31}f_{32}}{L}\bigg) - \dfrac{(f_{31}^2 - f_{11}^2)f_{22}}{L} = 0. 
\end{split}
\end{equation}
Differentiating  this equation with respect to $z_{k}$ leads to 
\begin{equation*}
L(f_{11}f_{12,z_{k-1}} - f_{31}f_{32,z_{k-1}}) -  H(f_{11}f_{32,z_{k-1}} - f_{31}f_{12,z_{k-1}}) = 0.
\end{equation*}
Substituting the expressions of  $H$ and $L$ in the latter, and after simplifying the expression, leads to 
\begin{equation*}
(f_{31}^2 - f_{11}^2) ( f_{11,z_0} f_{32,z_{k-1}} - f_{31,z_0} f_{12,z_{k-1}}) = 0. 
\end{equation*}
Note that $f_{31}^2 - f_{11}^2$ is not a constant, otherwise $H=0$. We conclude then that 
\begin{equation}\label{eqf}
f_{11,z_0} f_{32,z_{k-1}} - f_{31,z_0} f_{12,z_{k-1}} = 0. 
\end{equation}

Note also that $f_{11,z_0}$ and $f_{31,z_0}$ cannot vanish simultaneously, 
 otherwise (\ref{std7}) is not satisfied. Moreover, $f_{12,z_{k-1}}$ and $f_{32,z_{k-1}}$ cannot simultaneously  vanish, otherwise $F$ is not a $k$-th order evolution equation. Moreover, (\ref{SEq1n}) and (\ref{SEq3n}) imply that $f_{11,z_0}\neq 0$ if, and only if, $f_{12,z_{k-1}}\neq 0$ and 
 $f_{31,z_0}\neq 0$ if, and only if,  
$f_{32,z_{k-1}}\neq 0$. Therefore, 
  if $f_{11,z_0}=0$  then  $f_{12,z_{k-1}}=0$,  
 $f_{31,z_0}\neq 0$ and $f_{32,z_{k-1}}\neq 0$. 
 Similarly, if $f_{31,z_0}=0$, then  $f_{32,z_{k-1}}=0$,  
 $f_{11,z_0}\neq 0$ and $f_{31,z_{k-1}}\neq 0$.

We now consider equations (\ref{Eq1-2n}) and (\ref{Eq2-2n}).
 Differentiating these equations with respect to $z_{k-1}$ leads to
\begin{eqnarray}\label{eq1final}
f_{12,z_{k-1}} a_x = (a-c) \eta f_{32,z_{k-1}} - 2b(f_{11} f_{32,z_{k-1}} - f_{31}f_{12, z_{k-1}}), \\\label{eq2final}
f_{12,z_{k-1}} b_x = (a-c)(f_{11} f_{32,z_{k-1}} - f_{31}f_{12, z_{k-1}}) + 2b \eta f_{32,z_{k-1}}. 
\end{eqnarray}

If $f_{12, z_{k-1}} = 0$, then  $f_{32,z_{k-1}} \neq 0$, and the system of equations (\ref{eq1final}) and (\ref{eq2final}) becomes
\begin{equation*}
\left(\begin{array}{cc}
\eta & -f_{11} \\
f_{11} & \eta \end{array}\right)
\left(\begin{array}{c}
a-c \\2b\end{array}\right) = 0.
\end{equation*}
The determinant $\eta^2 + f_{11}^2 \neq 0$, and hence $a-c = b = 0$, which runs into a contradiction with the Gauss equation (\ref{Gauss}). 

If $f_{12, z_{k-1}} \neq 0$, then the system of equations (\ref{eq1final}) and (\ref{eq2final}) can be rewritten as follows:
\begin{eqnarray}\label{eq1final2}
a_x = (a-c)\eta \dfrac{f_{32,z_{k-1}}}{f_{12, z_{k-1}}} - 2b \dfrac{f_{11}f_{32,z_{k-1}} - f_{31}f_{12,z_{k-1}}}{f_{12, z_{k-1}}},\\\label{eq2final2}
b_x = (a-c)\dfrac{f_{11}f_{32,z_{k-1}} - f_{31}f_{12,z_{k-1}}}{f_{12, z_{k-1}}} + 2b \eta \dfrac{f_{32,z_{k-1}}}{f_{12, z_{k-1}}}.  
\end{eqnarray}
From (\ref{eqf}) and the assumption $f_{12, z_{k-1}} \neq 0$, we have 
\begin{equation*}
\dfrac{f_{32, z_{k-1}}}{f_{12,z_{k-1}}} = \dfrac{f_{31,z_0}}{f_{11,z_0}}. 
\end{equation*}
Substituting the latter in  (\ref{eq1final2}) and (\ref{eq2final2}) leads to 
\begin{eqnarray}\label{eq1final3}
a_x = (a-c)\eta \dfrac{f_{31,z_0}}{f_{11,z_0}} - 2b \dfrac{L}{f_{11,z_0}},\\\label{eq2final3}
b_x = (a-c)\dfrac{L}{f_{11,z_0}} + 2b \eta \dfrac{f_{31,z_0}}{f_{11,z_0}}.
\end{eqnarray}
Differentiating the latter two equations with respect to $z_0$ leads to 
\begin{equation*}
\left(\begin{array}{cc}\eta (f_{31,z_0}/f_{11,z_0})_{z_0}& -(L/f_{11,z_0})_{z_0} \\ (L/f_{11,z_0})_{z_0}& \eta (f_{31,z_0}/f_{11,z_0})_{z_0}\end{array}\right)\left(\begin{array}{c}a-c \\2b\end{array}\right) =0.
\end{equation*}
The determinant $\eta^2 (f_{31,z_0}/f_{11,z_0})_{z_0}^2 + (L/f_{11,z_0})_{z_0}^2=0$. Otherwise, $a-c = b = 0$, which runs into a contradiction with the Gauss equation (\ref{Gauss}). Threfore, $ (f_{31,z_0}/f_{11,z_0})_{z_0} =0$, and $(L/f_{11,z_0})_{z_0} = 0$. Note that the vanishing of  $(f_{31,z_0}/f_{11,z_0})_{z_0}$ means that 
\begin{equation}\label{f31}
f_{31} = \gamma f_{11} + \mu,
\end{equation}
where $\gamma$ and $\mu$ are constants. We have then
\begin{equation*}
\dfrac{f_{32,z_{k-1}}}{f_{12,z_{k-1}}} = \dfrac{f_{31,z_0}}{f_{11,z_0}} = \gamma
\end{equation*}
and hence
\begin{equation}\label{f32}
f_{32} = \lambda f_{12} + \nu,
\end{equation}
where $\nu$ depends on $z_0,..., z_{k-2}$. Note that (\ref{neq2}) and $f_{12,z_{k-1}} \neq 0$ means that $\mu \neq 0$. This fact can also be obtained from the non-vanishing of 
\begin{equation}\label{Lmu}
L = -\mu f_{11,z_0} \neq 0.
\end{equation}
 In light of (\ref{f31}), (\ref{f32}) and (\ref{Lmu}), equations (\ref{eq1final3}) and (\ref{eq2final3}) become
\begin{eqnarray*}\label{eq1f4}
a_{x} = (a-c)\eta \gamma + 2b\mu,\\\label{eq2f4}
b_x = -(a-c) \mu + 2b\eta \gamma.
\end{eqnarray*}
We have also 
\begin{eqnarray*}\label{eq1final4}
f_{11}f_{32} - f_{31}f_{12} = \nu f_{11} - \mu f_{12},\\\label{eq2final4}
\eta f_{32} - f_{31}f_{22} = \eta \gamma f_{12} + \eta \nu - \gamma f_{11} f_{22} - \mu f_{22}.  
\end{eqnarray*}
Substituting 
the last four equations 
in (\ref{Eq1-2n}) and (\ref{Eq2-2n}) leads to 
\begin{eqnarray}\label{eq1N2}
f_{11} a_t + \eta b_t - (a-c)(\gamma f_{11}f_{22} -\eta\nu) - 2b(\eta \gamma f_{22} + \nu f_{11}) = 0,\\\label{eq2N2}
f_{11}b_t + \eta c_t - f_{22}c_x + (a-c)\nu f_{11} + 2b(\eta \nu - \gamma  f_{11} f_{22} - \mu f_{22}) = 0.
\end{eqnarray}

If $k\geq 3$, differentiating (\ref{eq1N2}) and (\ref{eq2N2}) with respect to $z_{k-2}$, and then dividing by $f_{22,z{_k-2}} \neq 0$, leads to 
\begin{eqnarray}\label{eq1final5}
(-2bf_{11}+(a-c)\eta)\frac{\nu_{,z_{k-2}}}{f_{22,z_{k-2}}} = 
\gamma( (a-c)f_{11}+2b\eta), \\ \label{eq2final5}
((a-c)f_{11}+2b\eta)\frac{\nu_{,z_{k-2}}}{f_{22,z_{k-2}}} =
2b\gamma f_{11} +c_x+2b\mu ).
\end{eqnarray}
Observe that $-2bf_{11}+(a-c)\eta$ and $(a-c)f_{11}+2b\eta$ cannot vanish simultaneously, since $f_{11,z_0}\neq 0$ and $(a-c)^2+b^2\neq 0$. Therefore, from (\ref{eq1final5}) and (\ref{eq2final5})
we get 
\[
\gamma [(a-c)^2+4b^2]f_{11}^2+[2b\eta\gamma(a-c)+2b(c_x+2b\mu)]f_{11} +
4\gamma b^2\eta^2-\eta(a-c)(c_x+2b\mu)=0. 
\]
Differentiating  twice with respect to $z_0$ leads to 
$\gamma=0$ and $c_x+2b\mu=0$. Hence, (\ref{eq1N2}) and (\ref{eq2N2}) reduce to 
\begin{eqnarray} \label{*1}
f_{11}b_t+\eta c_t + [(a-c)f_{11}+2b\eta]\nu=0,\\ \label{*2}
f_{11}a_t+\eta b_t +[-2bf_{11}+(a-c)\eta]\nu=0.
\end{eqnarray}
It follows from these equations that $\nu$ depends at most on $z_0$ and 
\[
-[2b b_t+(a-c)a_t]f_{11}^2-2b\eta (c_t+a_t) f_{11}+\eta^2[(a-c)c_t-2bb_t]=0.
\]
Therefore \quad $2b b_t+(a-c)a_t=b(c_t+a_t)= (a-c)c_t-2bb_t=0$. \quad 
If $b=0$, then $a_t=0$ and hence $c_t$=0.  If $b\neq 0$, then $c_t+a_t=0$ 
and using the derivative of the Gauss equation (\ref{Gauss}), we get that 
$(a-c)c_t-2bb_t=-2a_t(a-c)=0$. If either $a_t=0$ or $a-c=0$,
we get $a_t=c_t=0$. Hence, for any $b$, we have that $a$ and $c$ do not depend on $t$. It follows from the Gauss 
equation that  $b$ also does not depend on $t$. Therefore (\ref{eq1N2}) and 
(\ref{eq2N2}) reduce to 
\[ 
\left(
\begin{array}{cc}
\eta\nu &-f_{11}\nu\\
f_{11}\nu & \eta \nu
\end{array}
\right)
\left(
\begin{array}{c}
a-c)\\
2b
\end{array}
\right)=0.
\]
Since $a-c$ and $b$ cannot vanish simultaneously, the determinant 
$\nu^2(\eta^2+f_{11}^2)=0$. hence $\nu=0$. 

The above analysis implies that $f_{31}=\mu\neq 0$, $f_{32}=0$, 
$H=f_{11}f_{11,z_0}$ and $L=-\mu f_{11,z_0}$. Therefore, (\ref{eqFmF})
reduces to 
\[
 (\mu^2 -f_{11}^2)(\eta f_{12}-f_{22}f_{11})=0.
 \]
Differentiating with respect to $z_{k-1}$, we get a contradiction 
since $(\mu^2-f_{11}^2)f_{12,z_{k-1}}\neq 0$. 

\vspace{.1in}

If $k=2$, the proof of this proposition was given in \cite{KKT1}. 
Finally, we conclude that for any $k\geq 2$,  whenever $HL \neq 0$, the system of equations (\ref{Eq1}), (\ref{Eq2}) and  (\ref{Gauss}) is inconsistent. 
\end{proof}


\begin{thebibliography}{KT95}
\bibitem 
{BealsRabeloTenenblat}
R.~Beals, M.~Rabelo and K.~Tenenblat, \emph{B\"acklund transformations and inverse scattering solutions for some pseudospherical surface equations}, Stud. Appl. Math. \textbf{81} (1989), no.~2, 125-151. 

\bibitem
{CastroKamran}
T. ~Castro Silva and N.~ Kamran, \emph{Third-order differential equations and local isometric immersions of pseudo-spherical surfaces}, Commun. Contemp. Math. \textbf{18} (2016), 1650021, 41 pp.

\bibitem 
{CastroTenenblat}
T. ~Castro Silva and K.~Tenenblat, \emph{Third order equations describing pseudo-spherical surfaces}, J. Diff. Eq. \textbf{259} (2015), 4897-4923.

\bibitem
{CatalanoOliveira16}
D. Catalano Ferraioli and L.A. ~de Oliveira, \textit{Second order evolution equations which describe pseudospherical surfaces}, J. Diff. Eq. \textbf{260} (2016), 8072-8108.


\bibitem
{CatalanoOliveira17}
D. Catalano Ferraioli and L.A. ~de Oliveira, \textit{Local isometric immersions of pseudospherical surfaces described by evolution equations in conservation law form}, J. Math. Anal. Appl. \textbf{446} (2017), 1606-1631.

\bibitem{FT} D. ~Catalano Ferraioli, and K. ~Tenenblat,  \textit{Fourth order evolution equations which describe pseudospherical surfaces}. J. Diff. Eq. \textbf{257} (2014), 3165-3199. 

\bibitem 
{CavalcanteTenenblat}
J.~A. Cavalcante and K.~Tenenblat, \emph{Conservation laws for nonlinear evolution equations}, J. Math. Phys. \textbf{29} (1988), no.~4, 1044-1049.

\bibitem 
{ChernTenenblat}
S.~S.~Chern and K.~Tenenblat, \emph{Pseudospherical surfaces and evolution
  equations}, Stud. {A}ppl. {M}ath \textbf{74} (1986), 55-83.

\bibitem 
{DingTenenblat}
Q.~Ding and K.~Tenenblat, \emph{On differential systems describing surfaces of constant curvature}, 
J. Diff. Eq. \textbf{184} (2002), 185-214.

\bibitem 
{FoursovOlverReyes}
V.M.~Foursov, P. J.~Olver, E.G.~Reyes, \emph{On formal integrability of evolution equations and local geometry of surfaces}, Differential Geom. Appl. \textbf{15} (2001), 183-199.

\bibitem 
{Gomes}
V.~P.~Gomes Neto, \emph{Fifth-order evolution equations describing pseudospherical surfaces}, 
J. Diff. Eq. \textbf{249} (2010), 2822-2865.

\bibitem 
{GorkaReyes}
P.~G\'orka and E.~G.~Reyes, \emph{The modified Hunter-Saxton equation}, J. Geom. Phys. \textbf{62} (2012), 1793-1809.

\bibitem 
{JorgeTenenblat} 
L.~Jorge and K.~Tenenblat, \emph{Linear problems associated to evolution equations of type 
$u_{tt}=F(u,u_x,u_{xx},u_t)$}, Stud. {A}ppl. {M}ath. \textbf{77} (1987), 103-117.

\bibitem 
{KKT1}
N.~Kahouadji, N.~Kamran and K.~Tenenblat, \emph{Second-order Equations and Local Isometric Immersions of Pseudo-spherical Surfaces}, Comm. Anal. Geom., \textbf{24}, (2016), no.~3, 605-643.


\bibitem 
{KKT2}
N.~Kahouadji, N.~Kamran and K.~Tenenblat, \emph{Local Isometric Immersions of Pseudo-spherical Surfaces and Evolution Equations}, Fields Inst. Commun., \textbf{75}, (2015), 369-381.

\bibitem 
{KamranTenenblat}
N.~Kamran and K.~Tenenblat, \emph{On differential equations describing
  pseudo-spherical surfaces}, J. Diff. Eq. \textbf{115}
  (1995), no.~1, 75-98.
  
 \bibitem 
{Rabelo88}
M.~Rabelo, \emph{A characterization of differential equations of type $u_{xt}=F(u, \partial u/\partial x,..., \partial^k/\partial x^k)$ which describe pseudospherical surfaces}, An. Acad. Brasil. Ci\^enc. \textbf{60} 
 (1988), no. ~2, 119-126. 

\bibitem 
{Rabelo89}
M.~Rabelo, \emph{On equations which describe pseudospherical surfaces}, 
Stud. {A}ppl. {M}ath. \textbf{81} (1989), 221-248.

\bibitem 
{RabeloTenenblat90}
M.~Rabelo and K.~Tenenblat, \emph{On equations of the type $u_{xt}=F(u,u_{x})$ which describe pseudospherical surfaces}, J. Math. Phys \textbf{6}  (1990), 1400-1407.

\bibitem 
{RabeloTenenblat92}
M.~Rabelo and K.~Tenenblat, \emph{A classification of pseudospherical surfaces equations of the type $u_{t}=u_{xxx}+G(u,u_{x},u_{xx})$ which describe pseudospherical surfaces}, J. Math. Phys \textbf{33}  (1992), 537-549.

\bibitem 
{Reyes98}
E.~G.~Reyes, \emph{Pseudo-spherical surfaces and integrability of evolution equations},  
J. Diff. Eq. \textbf{147}
  (1998), no.~1, 195-230.

\bibitem 
{Reyes00}
E.~G.~Reyes, \emph{Conservation laws and Calapso-Guichard deformations of equations describing pseudo-spherical surfaces}, J. Math. Phys. \textbf{41} (2000), 2968-2989.     
  
\bibitem 
{Reyes02}
E.~G.~Reyes, \emph{Geometric integrability of the Camassa-Holm equation}, Lett. Math. Phys. \textbf{59} 
(2002), 117-131.  

\bibitem 
{Reyes06}
E.~G.~Reyes, \emph{Pseudo-potentials, nonlocal symmetries and integrability of some shallow water equations}, Selecta Math. (N.S.) \textbf{12} (2006), 241-270.

\bibitem 
{Reyes106}
E.~G.~Reyes, \emph{Correspondence theorems for hierarchies of equations of pseudo-spherical type}, 
J. Diff. Eq. \textbf{225} (2006), 26-56.
  

\end{thebibliography}

\vspace{.35in}

\noindent

\address{Department of Mathematics, Northeastern Illinois University\\
Chicago, IL 60625-4699, USA\\
\email{n-kahouadji@neiu.edu}}

\address{Department of Mathematics and Statistics, McGill University\\
Montreal, Quebec, H3A 0B9, Canada\\
\email{nkamran@math.mcgill.ca}}

\address{Department of Mathematics, Universidade de Bras\'\i lia\\
Bras\'\i lia -- DF, 70910-900, Brazil\\ 
\email{K.Tenenblat@mat.unb.br}}
\end{document}